\documentclass[11pt]{amsart}

\usepackage{amsmath,amssymb,amsthm}
\usepackage{verbatim}
\usepackage[english]{babel}
\usepackage[utf8]{inputenc}
\usepackage[T1]{fontenc}
\usepackage{comment}
\DeclareMathAlphabet{\mathbbo}{U}{bbold}{m}{n}
\usepackage{graphicx}

\newcommand{\E}{\mathbb{E}}
\newcommand{\R}{\mathbb{R}}
\newcommand{\N}{\mathbb{N}}
\newcommand{\Z}{\mathbb{Z}}
\renewcommand{\P}{\mathbb{P}}
\newcommand{\1}{\mathbbo{1}}
\renewcommand{\S}{S}
\renewcommand{\epsilon}{\varepsilon}
\renewcommand{\liminf}{\underline{\lim}}
\renewcommand{\limsup}{\overline{\lim}}

\newcommand{\miniop}[3]{%
\renewcommand{\arraystretch}{0.6}
\begin{array}{c}
{\scriptstyle #1}\\
#2\\
{\scriptstyle #3}
\end{array}
\renewcommand{\arraystretch}{1}}

\author{Olivier Garet}
\title[A Central Limit Theorem]{A Central Limit Theorem for the number of descents and some urn models}
\newtheorem{theorem}{Theorem}[section]
\newtheorem{lemma}[theorem]{Lemma}
\begin{document}

\address{Université de Lorraine, CNRS, IECL, F-54000 Nancy, France
}
\email{Olivier.Garet@univ-lorraine.fr}

\def\motsclefs{Central Limit Theorem, urn models.}

\subjclass[2000]{60F05, 60C05.} 
\keywords{\motsclefs}

\begin{abstract}
The purpose of this work is to establish a Central Limit Theorem that can be applied to a particular form of Markov chains, including
  the number of descents in a random permutation of $\mathfrak{S}_n$, two-type generalized P\'olya urns, and some other urn models.
  \end{abstract}

\maketitle

The purpose of this work is to establish a Central Limit Theorem that can be applied to a particular form of Markov chains, a form that is encountered in various problems of a combinatorial nature.
This class is large enough to cover a number of classical results, such as the Central Limit Theorem for the number of descents in a random permutation of $\mathfrak{S}_n$, two-type generalized P\'olya urns, as well as some other urn models.

We can immediately state the theorem:

\begin{theorem}\label{lbt}
  Let $(a_n)_{n\ge 1}$ be a sequence of integrable random variables.
  We note, for $n\ge 1$, 
  $S_n=a_1+\dots+a_{n}$.
  We suppose that there exists some constant $M$ and, for $k\in\{1,2,3\}$, non-random sequences  $(\alpha^n_{k})_{n\ge 1}$, $(D^n_{k})_{n\ge 1}$ and real numbers $\alpha_1,\alpha_2,\alpha_3,D_1,D_2,D_3$ such that
  \begin{itemize}
  \item $\forall n\ge 1\quad |a_{n}|\le M$.
  \item $\forall n\ge 1\quad \forall k\in\{1,2,3\}\quad \E[a_{n+1}^k|a_1,\dots,a_n]=D_k^n-\frac{\alpha_k^n}{n}S_n$, with\\ $\displaystyle\lim_{n\to +\infty} \alpha_k^n=\alpha_k$ and $\displaystyle\lim_{n\to +\infty} D_k^n=D_k$.
  \item $\alpha^n_1=\alpha_1+o\Bigl(\frac1{\sqrt n}\Bigr)$ and  $D^n_1=D_1+o\Bigl(\frac1{\sqrt n}\Bigr)$.   
  \end{itemize}
  We also compute  $\ell=\frac{D_1}{\alpha_1+1}$ and $D=D_2-\ell(\ell+\alpha_2)$.
  
  If it is then verified that $\alpha_1>-\frac12$ and $D>0$, then

  $\frac{S_n-n\ell}{\sqrt{n}}$ converges in distribution to $\mathcal{N}(0,S)$ with $S=\frac{D}{2\alpha_1+1}$.   
\end{theorem}

The article is organized as follows: in a first part, as an appetizer, we recall the definitions, the Central Limit Theorem on the number of descents and we see how to deduce it from our general theorem.
The second part is devoted to the proof of the theorem, while the last two parts give applications to urn models, including the P\'olya urn model, but not limited to it.

\section{The number of descents}

When $\sigma$ is a permutation of the set $\{1,\dots,n\}$, the descents of $\sigma$ are the points
\begin{align*}
  Desc(\sigma)&=\{i\in\{1,\dots,n-1\}; \sigma(i)>\sigma(i+1)\}.
\end{align*}
Similarly, the ascents are defined by
\begin{align*}
  Asc(\sigma)&=\{i\in\{1,\dots,n-1\}; \sigma(i)<\sigma(i+1)\}.
\end{align*}

It is well known that the number of descents (or the number of ascents)
of a uniformly chosen random permutation has an asymptotic normal behaviour:
if $\sigma_n$ is chosen uniformly among the permutations on $\{1,\dots,n\}$, we have
$$\frac{|Desc(\sigma_n)|-n/2}{\sqrt{n}}\Longrightarrow\mathcal{N}\Bigl(0,\frac1{12}\Bigr).$$
In a recent article~\cite{MR3741694}, Chatterjee and Diaconis counted no less than 6 proofs. Thus, we venture here to present a seventh -- and perhaps even an eighth, if we count as such the transcription of the question into an urn problem, which can be studied using this article or the literature.
Note also the recent proof by \"Ozdemir using martingales~\cite{ozdemir}.
Our proof of Theorem~\ref{lbt}  is based on conditional expectation together with the method of moments.
Among the six proofs for the CLT on the number of descents evoked by Chatterjee and Diaconis, the proof by David and Barton~\cite{MR0155371} (Chapter 10) also relies on the method of moments, however they use the generating functions whereas our approach is more related to conditioning.

An easy way to perform the uniform distribution on the permutations in $\{1,\dots,n\}$ is to apply the insertion algorithm: take $\sigma_1$ as the identity on  $\{1\}$, then
$\sigma_{n+1}=\sigma_n \circ (n+1\quad  U_{n+1})$, where $(U_n)$
is a sequence of independent random variables such that for each $n$, $U_n$ follows the uniform distribution on  $\{1,\dots,n\}$.
Then, for each  $n$, $\sigma_n$ follows the uniform distribution on  $\mathfrak{S}_n$.

Let $D_n=|Desc(\sigma_{n})|$. We have, for $n\ge 1$:
$$D_{n+1}=D_n+\1_{E_{n+1}},\text{ with } E_{n+1}=\{U_{n+1}-1\in Asc(\sigma_{n})\cup\{0\}\}.$$
Let also $\mathcal{F}_n=\sigma(U_1,\dots,U_{n})$, then
\begin{align}
  \label{dynamique}
  \P(E_{n+1}|\mathcal{F}_n)&=\frac{|Asc(\sigma_{n})|+1}{n+1}=\frac{(n-1-D_n)+1}{n+1}=\frac{n-D_n}{n+1}
\end{align}
Let $S_n=D_{n}-\frac{n-1}2$. Let us define $S_1=0$,  $a_1=S_1$, and for $n\ge 1$: 
$$a_{n+1}=S_{n+1}-S_{n}=(D_{n+1}-D_{n})-\frac12=\1_{E_{n+1}}-\frac12,$$
then $S_n=a_1+\dots+a_n$.
The random variable $a_n$ is $\{-1/2,1/2\}$ valued.

Finally,
\begin{align}\E(a_{n+1}|\mathcal{F}_{n})&=\P(E_{n+1}|\mathcal{F}_n)-\frac12
  \label{eqdesc}
  =\frac{\frac{n-1}2-D_n}{n+1}=-\frac{S_{n}}{n+1}.
  \end{align}

Theorem~\ref{lbt} applies particularly easily: we have $\mathcal{F}_n=\sigma(a_1,\dots,a_n)$, $(a_n)$ is uniformly bounded by  $1/2$, Equation~\eqref{eqdesc} gives the second assumption for $k=1$, $k=2$ and $k=3$ easily follow because the sequence $(a_n^2)$ is constant equal to $1/4$: we can apply the theorem with $M=1/2$, $D_1=0$, $D_2=1/4$, $D_3=0$ and $\alpha_1=1$, $\alpha_2=0$, $\alpha_3=1/4$, which gives the announced result.

\section{Proof of the Theorem}We already note  $\mathcal{F}_n=\sigma(a_1,\dots,a_n)$. Let's start with the study of the first moment.

\begin{align*}
  \E(S_{n+1})&=\E(S_n)+\E(a_{n+1})=\E(S_n)+D^n_1-\frac{\alpha^n_1}n\E(S_n)\\
  &=(1-\frac{\alpha^n_1}n)\E(S_n)+D^n_1\\
  &=(1-\frac{\alpha_1}n)\E(S_n)+D_1+o\Bigl(\frac{1}{\sqrt{n}}\Bigr)\text{ since }|\E(S_n)|\le nM
\end{align*}
To solve the recurrence, we rely on two analytic lemmas:

\begin{lemma}\label{kconstant}
  Let $C,k,\alpha,\beta\in\R$ with $\alpha-1<\beta\le \alpha$, $\beta+k>-1$ and  $(u_n)_{n\ge n_0}$,  sequences such that
  \begin{align*}
    u_{n+1}&=(1-\frac{k}{n+c})u_n+Cn^{\alpha}+o(n^{\beta}),
  \end{align*}
  
  \end{lemma}

\begin{proof}
  Note that
  \begin{align*}
    1-\frac{k}{n+c}=\frac{n+c-k}{n+c}=\frac{\Gamma(n+c-k+1)}{\Gamma(n+c-k)}\frac{\Gamma(n+c)}{\Gamma(n+c+1)}=\frac{P_{k,c}(n)}{P_{k,c}(n+1)},
  \end{align*}
  with $P_{k,c}(x)=\frac{\Gamma(x+c)}{\Gamma(x+c-k)}$.\\
  By the Stirling formula, $P_{k,c}(x)= x^k(1+O(1/x))$ at infinity, the recurrence equation is then rewritten
\begin{align*}
  P_{k,c}(n+1)u_{n+1}&=P_{k,c}(n)u_n+C P_{k,c}(n+1)n^{\alpha}+o(n^{\beta+k})\\
  &=P_{k,c}(n)u_n+C n^{\alpha+k}+o(n^{\beta+k}),
  \end{align*}
  and then, with the telescopic sum,
  $u_nP_{k,c}(n)=C\frac{n^{\alpha+k+1}}{\alpha+k+1}+o(n^{\beta+k+1}),$
  and finally
  $$u_n=\frac{C}{\alpha+k+1}n^{\alpha+1}+o(n^{\beta+1}).$$
  \end{proof}

  \begin{lemma}\label{kpositif}
  Let $C,k,\alpha\in\R$ with  $C>0$, $\alpha+k>-1$ and  $(u_n)_{n\ge n_0}$, $(k_n)_{n\ge n_0}$ sequences such that
  \begin{align*}
    u_{n+1}&=(1-\frac{k_n}{n+c})u_n+Cn^{\alpha}+o(n^{\beta}),\text{ with }\lim k_n=k.
  \end{align*}

  We suppose moreover that $\forall n\ge n_0\quad u_n\ge 0$.
  
Then, 
  $$u_n=\frac{C}{\alpha+k+1}n^{\alpha+1}+o(n^{\alpha+1}).$$  

\end{lemma}

\begin{proof} 
  One can find $n_0$ such that $n+c>0$, $v_n>0$, and $1-\frac{k_n}{n+c}>0$ for $n\ge n_0$, then for $\epsilon>0$ small enough, take $n_1\ge n_0$ such that
  $1-\frac{k-\epsilon}{n+c}>1-\frac{k_n}{n+c}>1-\frac{k+\epsilon}{n+c}>0$ for $n\ge n_1$.
  Let $w_{n_1}=x_{n_1}=u_{n_1}$, then for $n\ge n_1$,
  $$w_{n+1}=w_n\left(1-\frac{k+\epsilon}{n+c}\right)+v_n
\text{ and }x_{n+1}=x_n\left(1-\frac{k-\epsilon}{n+c}\right)+v_n.
  $$ By natural induction,
   $x_n\ge u_n\ge w_n\ge 0$ for each $n\ge n_1$, then
$$\miniop{}{\liminf}{n\to +\infty}u_n n^{-(\alpha+1)}\ge  \miniop{}{\liminf}{n\to +\infty}w_n n^{-(\alpha+1)}=\frac{C}{\alpha+k+\epsilon+1},$$
where the last equality follows from Lemma~\ref{kconstant}.
Letting $\epsilon$ tend to  $0$, we get    $\miniop{}{\liminf}{n\to +\infty}u_n n^{-(\alpha+1)}\ge  \frac{C}{\alpha+k+1}$.
Similarly
$$\miniop{}{\limsup}{n\to +\infty}u_n n^{-(\alpha+1)}\le  \miniop{}{\limsup}{n\to +\infty}x_n n^{-(\alpha+1)}=\frac{C}{\alpha+k-\epsilon+1},$$
and, letting $\epsilon$ tend to  $0$,    $\miniop{}{\limsup}{n\to +\infty}u_n n^{-(\alpha+1)}\le  \frac{C}{\alpha+k+1}$, which completes the proof.
\end{proof}
Applying Lemma~\ref{kconstant} to the sequence $(\E(S_n))_{n\ge 1}$, we get $\E(S_n)=n\ell+o(\sqrt{n})$, with $\ell=\frac{D_1}{\alpha_1+1}$.

By the Slutsky Lemma, it is now sufficient to prove that $\frac{S_n-\E(S_n)}{\sqrt{n}}\implies \mathcal{N}(0,D/(2\alpha_1+1))$.

We will first reduce to the case where $D_1^n=0$ for all $n$.
Let $\ell_n=\E(a_n)$ and $a'_n=a_n-\ell_n$.

Since $\E[a_{n+1}|a_1,\dots,a_n]=D_1^n-\frac{\alpha_1^n}{n}S_n$,
we have $\ell_n=D_1^n-\frac{\alpha_1^n}{n}\E(S_n)$,
so $\ell_n\to D_1-\alpha_1\ell=(\alpha_1+1)\ell-\alpha_1\ell=\ell$.

It is clear that $(a'_n)$
checks the same sort of assumptions as $(a_n)$, only the limit values are changed.
If we put $f_{k,n}(x)=D^n_k-\frac{\alpha^n_k}n x$, we have
$\E(a_{n+1}|\mathcal{F}_n)=f_{1,n}(S_n)$, then  $\ell_{n+1}=f_{1,n}(\E(S_n))$, and
\begin{align*}
  \E(a'_{n+1}|\mathcal{F}_n)&=f_{1,n}(S_n)-f_{1,n}(\E(S_n))=-\frac{\alpha_1^n}{n}(S_n-\E(S_n))
\end{align*}
which gives the constants associated to $a'_n$ : $\alpha'_1=\alpha_1$ and $D'^n_1=0$ for all $n$.
We also have
\begin{align*}
  \E( a'^2_{n+1}|\mathcal{F}_n)&=\E(a_{n+1}^2+\ell^2_n-2\ell_n a_{n+1}|\mathcal{F}_n)\\
  &=f_{2,n}(S_n)+\ell^2_n-2\ell_n f_{1,n}(S_n)\\
  &=f_{2,n}(S'_n)-\alpha_2^n\frac{\ell_1+\dots+\ell_n}n +\ell_n^2-2\ell_n f_{1,n}(S'_n)+2\alpha_1^n\ell_n\frac{\ell_1+\dots+\ell_n}n\\
  &=f'_{2,n}(S'_n),
\end{align*}  
with $f'_{2,n}(x)=f_{2,n}(x)-\alpha_2^n\frac{\ell_1+\dots+\ell_n}n +\ell^2_n-2\ell_n f_{1,n}(x)+2\alpha_1^n\ell_n\frac{\ell_1+\dots+\ell_n}n$, which leads to
\begin{align*}D'_2&=\lim f'_{2,n}(0)=D_2-\alpha_2\ell +\ell^2-2\ell D_1+2\alpha_1\ell^2\\
  &=D_2-\ell(\alpha_2-\ell+2D_1-2\alpha_1\ell)\\
  &=D_2-\ell(\alpha_2+\ell-2(D_1-\ell(1+\alpha_1)))\\
  &=D_2-\ell(\ell+\alpha_2)
\end{align*}

The same kind of calculation applies for $k=3$, it is not detailed because the explicit values of the corresponding constants are not necessary.

Then, we can assume without loss of generality that $D_1^n=0$ and $\E(S_n)=0$ for all $n$.
We will now estimate by recurrence the sequences of the different moments, for the orders greater than one.

For each non-negative integer $k$, there exist polynomials $R_k$ and $T_k$
such that $R_k(x,y)=\frac{k(k-1)(k-2)}{6}x^{k-3}y^3+T_k(x,y)$

$$(x+y)^k=x^k+kx^{k-1}y+\frac{k(k-1)}2 x^{k-2}y^2 +R_k(x,y),$$
with
$$\forall x\in\R\quad \forall y\in [-M,M]\quad |R_k(x,y)|\le (M+1)^k(1\vee |x|^{k-3})$$
$$\forall x\in\R\quad \forall y\in [-M,M]\quad |T_k(x,y)|\le (M+1)^k(1\vee |x|^{k-4})$$
Particularly, we get
$$\S_{n+1}^k=\S_n^k+k \S_n^{k-1}a_{n+1}+\frac{k(k-1)}2 \S_n^{k-2}a_{n+1}^2+ R_k(\S_n,a_{n+1}).$$
Conditioning by $\mathcal{F}_n$, then reintegrating, we have
\begin{align*}
  \E(\S_{n+1}^k)=&(1-\frac{k\alpha_1^n}{n})\E(\S_n^k)+\frac{D_2^n k(k-1)}2\E(\S_n^{k-2})\\
  &+(kD_1^n-\frac{k(k-1)}2\frac{\alpha_2^n}n)\E(S_n^{k-1})+\E( R_k(\S_n,a_{n+1}))\\
\end{align*}
so, since $D_1^n=0$, we have
\begin{align}
  \label{larec}
\nonumber \E(\S_{n+1}^k) &=(1-\frac{k\alpha_1^n}{n})\E(\S_n^k)+\frac{D_2^n k(k-1)}2\E(\S_n^{k-2})\\&+\E( R_k(\S_n,a_{n+1}))+O(\frac1{n}\E(S_n^{k-1})).
\end{align}

For $k=2$, we have $R_2=0$ and $\E(S_n)=0$, which simply gives
\begin{align*}
  \E(\S_{n+1}^2)=(1-\frac{2\alpha_1^n}{n})\E(\S_n^2)+D_2^n.
\end{align*}
With Lemma~\ref{kpositif}, we get
\begin{align*}
  \E(\S_{n}^2)=D'n+o(n),\text{ where }D'=\frac{D}{2\alpha_1+1}.
  \end{align*}

Let's move on to the general case.
With the help of the recursion~\eqref{larec}, we will prove that for each $k\in\N$, we have 
\begin{align}\label{afr}\E(\S_n^k)=C_k n^{k/2}+o(n^{k/2}),
\end{align}
where $(C_k)_{k\ge 0}$ is the sequence defined by $C_0=1$, $C_1=0$ and $C_{k}=D'C_{k-2}(k-1)$: since $D'>0$, this is the sequence for the moments of $\mathcal{N}(0,D')$.
We note that terms with even index are positive. This will be useful to apply Lemma~\ref{kpositif}.

The proof is by induction on $k$.
The asymptotic expansion is proved for $k=0,1,2$.
Assuming~\eqref{afr} holds for $\{0,1\dots,k-1\}$ and let us prove for $k$.
We invoke the inductive assumption for $k-2$ and $k-1$; it comes
\begin{align}
  \label{larecdeux}
    \E(\S_{n+1}^k)&=(1-\frac{k\alpha_1^n}{n})\E(\S_n^k)+C_{k-2}\frac{D_2k(k-1)}2 n^{k/2-1}\\\nonumber &+\E(R_k(\S_n,a_{n+1}))+o(n^{k/2-1}).
  \end{align}
  If $k$ is odd, then
  $$|\E(R_k(\S_n,a_{n+1}))|\le (M+1)^k (1+\E(\S_n^{k-3}))=O(n^{\frac{k-3}2}).$$
  If $k$ is even, then
\begin{align}
  |\E(T_k(\S_n,a_{n+1}))|\le (M+1)^k (1+\E(\S_n^{k-4}))=O(n^{\frac{k-4}2}),
\end{align}
whereas
\begin{align*}
  \E(\S_n^{k-3}a_{n+1}^3)&=\E(\S_n^{k-3}(D^n_3-\frac{\alpha_3^n\S_n}{n}))=O(n^{\frac{k-3}2}),
\end{align*}
which gives $\E(R_k(\S_n,a_{n+1}))=O(n^{\frac{k-3}2})$. In any case, we have thus
\begin{align}
  \label{larecd}
    \E(\S_{n+1}^k)&=(1-\frac{k\alpha^n_1}{n})\E(\S_n^k)+C_{k-2}\frac{D_2k(k-1)}2 n^{k/2-1}+o(n^{k/2-1}).
\end{align}

If $k$ is odd, we have $|\E(S_n^{k})|\le \E(S_n^{k-1} |S_n|)\le Mn\E(S_n^{k-1})=O(n^{1+(k-1)/2})$, so
\begin{align}
    \E(\S_{n+1}^k)&=(1-\frac{k\alpha_1}{n})\E(\S_n^k)+C_{k-2}\frac{D_2k(k-1)}2 n^{k/2-1}+o(n^{k/2-1}),
\end{align}
and we apply Lemma~\ref{kconstant}.
Otherwise, the sequence $\E(S_n^k)$ is non-negative and we apply Lemma~\ref{kpositif}.

In both cases, it comes that
\begin{align*}
  \E(\S_{n}^k)&=\frac{2}{(1+2\alpha_1)k}\frac{D_2C_{k-2} k(k-1)}2 n^{k/2}+o(n^{k/2})=C_k n^{k/2-1}+o(n^{k/2-1}),
\end{align*}
which completes the inductive step.
We have proved that for each $k\ge 0$, we have
$$\lim_{n\to +\infty} \E((\frac{\S_n}{\sqrt{n}})^k)=\E(X^k),$$
where $X\sim\mathcal{N}(0,D')$.
By the method of moments, we conclude that
$$\frac{\S_n}{\sqrt{n}}\Longrightarrow\mathcal{N}(0,D').$$

\section{P\'olya's urn with random coefficients}
There is a wealth of literature on P\'olya's urns and their various extensions. For the properties we will find here, the most natural references are Freedman~\cite{MR0177432}, Bagchi and Pal~\cite{MR791169}, Smythe~\cite{MR1422883}, and Janson~\cite{MR2040966}.
We work here with ballot boxes with two types of balls. The model considered is a balanced model, i.e. at each step, the number of balls in the urn increases by $N$, where $N$ is a natural non-zero fixed integer.

Let's describe the dynamics: the urn initially contains $a_0$ white balls and $b_0$ black balls. The composition of the urn varies over time as follows:
at time $n$, a ball is shot into the urn, if it is white, it is put back into the urn as well as $B^1_n$ white and $N-B^1_n$ black, where, conditionally to the above, $B^1_n$ follows the law $\mu_1$.
If it is black, it is put back in the box along with $B^2_n$ white and $N-B^2_n$ black, where, conditionally to the above, $B^2_n$ follows the law $\mu_2$.
We only impose $-1\le B^1_n\le N$ and $-1\le N-B^2_n\le N$: it is therefore possible to remove a ball, but only if it has been fired.

At time $n$, the urn therefore contains $a_0+b_0+nN$ balls.
The number of white balls in the urn at the time $n$ is $a_0+\dots+a_n$, with

$$\mathcal{L}(a_{n+1}|a_0,\dots,a_n)=\frac{a_0+S_n}{a_0+b_0+nN}\mu_1 +\frac{b_0+(nN-S_n)}{a_0+b_0+nN}\mu_2.$$

Noting $m_{i,k}=\int_{\R} x^k \ d\mu_i$, we have then
\begin{align*}
  \E(a_{n+1}^k|a_0,\dots,a_n)=\frac{a_0+S_n}{a_0+b_0+nN}m_{1,k} +\frac{b_0+(nN-S_n)}{a_0+b_0+nN}m_{2,k}.
\end{align*}
Assumptions about the form of conditional expectations are easily verified, with $D_k=m_{2,k}$
and $\alpha_k=\frac{m_{2,k}-m_{1,k}}N$.

The first condition of validity of the Central Limit Theorem is given by
$\alpha_1>-\frac12$, or $\rho=\frac{m_{1,1}-m_{2,1}}N<\frac12$.
In the classic formalism of P\'olya's ballot boxes,
the mean replacement matrix is

$$\begin{pmatrix}m_{1,1} &N-m_{1,1}\\ m_{2,1} & N-m_{2,1}\end{pmatrix},$$
whose eigenvalues are $N$ and $m_{1,1}-m_{2,1}$.
The condition $\rho<\frac12$ thus found is indeed the classic condition on the eigenvalue quotient, which is classically denoted as the "small urns" case.
The second condition $D>0$ requires a little more calculations. Letting
$r_i=\frac{m_{i,1}}N$ and $v_i=\text{Var} \frac{B^i_1}N=\frac{m_{i,2}^2}{N^2}-r_i^2$, we find
$\frac{D}{2\alpha_1+1}=\frac{N^2(R+S)}{(2\alpha_1+1)(\alpha_1+1)^2}$ with $$R=r_2\alpha_1^2(1-r_1)\text{ and }S=(\alpha_1+1)(v_2(1-r_1)+v_1r_2).$$
Since $r_2\ge 0$ and $r_1\le 1$, $R$ and $S$ are non-negative.
It is then easy to see that only a few degenerate cases lead to $D=0$: the cases $\mu_1=\delta_N$ and $\mu_2=\delta_0$, as well as the case where there is $k$ such as $\mu_1=\mu_2=\delta_k$.
Thus, in the two-ball model, we recovers Smythe's CLT~\cite{MR1422883} on "small urns". It is interesting to note that we have an explicit form for the covariance of the normal limit distribution.

\subsection{Example: a removed ball and added balls.}
In this example, at each step, one ball is removed from the urn, then black and white balls are added, $b$ in total (with $b\ge 2$), the number of white balls following a distribution $\mu$ on $\{0,\dots,b\}$, with $\mu\not\in\{\delta_0,\delta_b\}$.
With the above notations, we have
$N=b-1$, $\mu_1=\mu*\delta_{-1}$ and $\mu_2=\mu$.
If we denote by $m$ the expectation of $\mu$ and by $\sigma^2$ its variance, we have then
$m_{1,1}=m-1$, $m_{1,2}=\sigma^2+(m-1)^2$, $m_{2,1}=m$, $m_{2,2}=\sigma^2+m^2$.
$\alpha_1=\frac{m_{1,1}-m_{2,1}}{N}=\frac{1}{b-1}$
and $\alpha_2=\frac{m_{2,2}-m_{1,2}}{N}=\frac{2m-1}{b-1}$
The quotient $\rho=\frac{m_{1,1}-m_{2,1}}{N}=-\frac{1}{b-1}$ checks $\rho<\frac12$. We have $D_1=m_{2,1}=m$, hence $$\ell=\frac{D_1}{1+\alpha_1}=\frac{D_1}{1-\rho}
=\frac{m}{1+\frac1{b-1}}=\frac{m(b-1)}b,$$ $D_2=m_{2,2}=\sigma^2+m^2$.
We finally have
$$D=D_2-\ell(\ell+\alpha_2)=\sigma^2+\frac{m}{b}(1-\frac{m}{b}).$$
So if $B_n$ denotes the number of white balls at time $n$
$$\frac{B_n-\frac{m(b-1)}b n}{\sqrt{n}}\Longrightarrow\mathcal{N}\left(0,\frac{b-1}{b+1}(\sigma^2+\frac{m}{b}(1-\frac{m}{b}))\right).$$
\subsection{Example: Bernard Friedman urn}
In this example, we add $\alpha$ balls of the drawn color and $\beta$ balls of the other color. We have thus
$N=\alpha+\beta$, $\mu_1=\delta_{\alpha}$ and $\mu_2=\delta_{\beta}$,
which gives us $D_1=m_{2,1}=\beta$, $D_2=m_{2,2}=\beta^2$,
$\alpha_1=\frac{m_{2,1}-m_{1,1}}N=\frac{\beta-\alpha}{\beta+\alpha}$
and $\alpha_2=\frac{m_{2,2}-m_{2,1}}N=\frac{\beta^2-\alpha^2}{\beta+\alpha}=\beta-\alpha$.
So $\ell=\frac{D_1}{\alpha_1+1}=\frac{\beta}{\frac{2\beta}{\alpha+\beta}}=\frac{\alpha+\beta}2$,
$D=D_2-\ell(\ell+\alpha_2)=\beta^2-\frac{\alpha+\beta}2\frac{3\beta-\alpha}2=\frac{(\alpha-\beta)^2}4$.
So if $3\beta>\alpha$ and $S_n$ is the number of white balls in the urn after $n$ steps, we have
$$\frac{S_n-\frac{\alpha+\beta}2n}{\sqrt{n}}\Longrightarrow\mathcal{N}\left(0,\frac{(\alpha-\beta)^2(\alpha+\beta)}{4(3\beta-\alpha)}\right).$$
In particular, if we add only one ball of the other color, we have
$\alpha=0$ and $\beta=1$ and
$$\frac{S_n-\frac{n}2}{\sqrt{n}}\Longrightarrow\mathcal{N}\Bigl(0,\frac1{12}\Bigr).$$

\subsection{Returning to the dynamics of descents}
The coincidence of this last equation with the CLT for downhill runs may lead to a new look to Equation~\eqref{dynamique}, which can be rewritten as
$$\mathcal{L}(\1_{E_{n+1}}|\mathcal{F}_n)=\frac{1+D_n}{1+n}\delta_0 +\frac{n-D_n}{1+n}\delta_1.$$
We can then see that the dynamics of variation of the descents is exactly the dynamics of the variation of the number of white balls in a Bernard Friedman urn starting from a white ball, and where we add to each step a ball of a different color from the one drawn. One can e.g. imagine that there is time $n$ $D_n+1$ white balls and $n-D_n$ black balls in the urn.
It is then obvious that the Central Limit Theorems obtained correspond. This is the eighth proof that was announced.

Although it is easy, this identification of the process of the number of descents with a Bernard Friedman urn seems new to us. We can extend the analogy even further. After one step, our Bernard Friedman urn necessarily contains a white ball and a black ball: this is the initial state of the urn considered by Diaconis and Fulton~\cite{MR1218674} to study the number $N_n$ of strictly negative points visited by the Internal Diffusion--Limited Aggregation model on $\Z$ after he has received $n$ points.
Remember that the  Internal Diffusion--Limited Aggregation model is defined as follows: start at time 1 with with one point at the origin and say it forms the first block.
Then, at each turn of the clock, start a walker from the origin and stop it as soon as it leaves the block, then enlarge the block by adding the stopped walker to it. At times $n$, the block is $\{-N_n,\dots,-N_n+(n-1)\}$, where $N_n$ denotes the number of points in the negative part.
The gambler's ruin theory says that one will exit from the left with probability $\frac{-N_n+(n-1)+1}{(-N_n+(n-1)+1)+(N_n+1)}=\frac{n-N_n}{n+1}$.
So, we get
$$\mathcal{L}(N_{n+1}-N_n|N_1,\dots,N_n)=\frac{1+D_n}{1+n}\delta_0 +\frac{n-D_n}{1+n}\delta_1$$
and the processes $(N_n)_{n\ge 1}$ and $(D_{n})_{n\ge 1}$ have the same distribution.

\section{Balls on a circle}
We conclude with an example that does not allow itself to be reduced to a P\'olya urn.
In this model, $b_0$ black balls and $w_0$ white balls are initially placed in a circle, with $1\le w_0<b_0$. The balls are arranged along the circle
to respect a placement rule which is described as follows:

First, white and black balls are placed alternately. At a given moment, when there are only black balls to be placed, they placed next to each other (see Figure~1).

Then, $3$ neighbouring balls are randomly drawn and removed; then $a$ white balls and $b$ black balls are inserted, in order to respect the placement rule.

To ensure that the rule can always be met, assumptions must be made about $a$ and $b$. The assumptions $a\ge 1$ and $a+b\ge 4$ are natural so that  there is always at least one white ball in the circle  and the number of balls grows.
We also need to add assumptions so that the blacks always stay ahead.
We can consider alternatively
\begin{enumerate}
\item $b-a\ge 3$.
\item $b-a=1$ and $w_0+b_0$ is odd.
\end{enumerate}

At the $n^{\text{th}}$ step, the urn contains $(a+b-3)n+w_0+b_0$ balls, including $S_n$ white.

Let's clarify the equations of the dynamics. The three balls that are removed can be spotted by the central ball. There are 4 possible configurations for this group of balls:
\begin{itemize}
\item a white ball surrounded by two black balls. This happens $S_n$ times, the gain for the whites is $a-1$; The advance of the blacks increases by $(b-a)-1\ge 0$.
\item a black ball surrounded by two white. This happens $S_n-1$ times, the gain for the whites is $a-2$; the advance of the blacks increases by $(b-a)+1\ge 2$.
\item two adjacent black balls and one white ball. It happens $2$ times if the black balls represent strictly more than half of the balls, which, as we will see, is ensured by the assumption that we have taken, the gain for the whites is $a-1$; the advance of the blacks increases by $(b-a)-1\ge 0$.
\item three black balls otherwise,the gain for the whites is $a$.
  The advance of the blacks increases by $(b-a)-3$.
  Under Hypothesis (1), $(b-a)\ge 3$, so the blacks are still ahead.
  Under Hypothesis (2), $b-a=1$, so the advance is reduced by two units.
  
However, the existence of this group of three balls means that Black's lead is at least 2; but the dynamics equations mean that the number of balls is always an odd number, so Black's lead is at least 3; losing two units means Black's lead is at least one.
\end{itemize}  

\begin{figure}
  \center
  \includegraphics[scale=0.9]{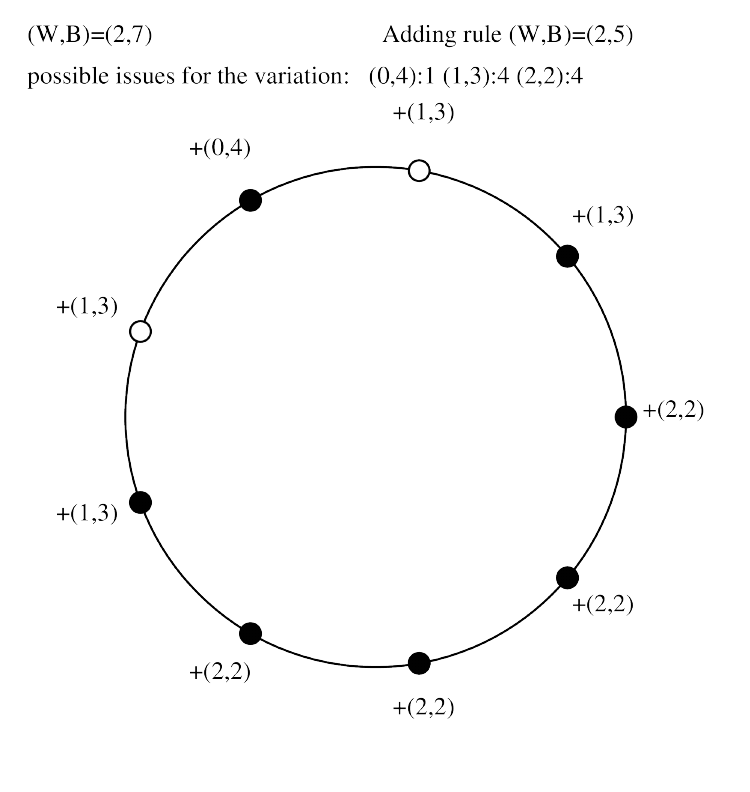}
  \caption{The dynamics, in a configuration with 2 white and 7 black balls. The transition rule is given by $a=2$ and $b=5$.}
\end{figure}  
If we write $a_{n+1}=S_{n+1}-S_n$, we have
\begin{align}
\label{bonneequation}
  \mathcal{L}(a_{n+1}|a_1,\dots,a_n)&=\frac{S_n-1}{(a+b-3)n+w_0+b_0}\delta_{a-2}+\frac{S_n+2}{(a+b-3)n+w_0+b_0}\delta_{a-1}\nonumber \\& +\left(1-\frac{2S_n+1}{(a+b-3)n+w_0+b_0}\right)\delta_a
\end{align}
Thus for any positive integer $k$, we have
\begin{align*}
 & \E(a_{n+1}^k|a_1,\dots,a_n)\\&=
  \frac{(a-2)^k(S_n-1)}{(a+b-3)n+w_0+b_0}+\frac{(a-1)^k(S_n+2)}{(a+b-3)n+w_0+b_0}\\&\quad +a^k-\frac{(2S_n+1)a^k}{(a+b-3)n+w_0+b_0}\\
  &=
  \frac{(a-2)^k+(a-1)^k-2a^k}{(a+b-3)n+w_0+b_0}S_n+a^k+\frac{-(a-2)^k+2(a-1)^k-1}{(a+b-3)n+w_0+b_0}.
 \end{align*} 
Conditional expectations are indeed of the required form.

  In particular $D_1=a$, $D_2=a^2$, $\alpha_1=\frac{3}{a+b-3}$, $\alpha_2=\frac{6a-5}{a+b-3}$.
  It is obvious that $\alpha_1>-\frac12$, but the fact that $D>0$ needs more care.

  Let $C=a+b$ and $x=\frac{a}{a+b}$. We have $C\ge 4$ and $x\in(0,\frac12)$.
  We get
  \begin{align*}
    D_1&=Cx;\quad \alpha_1=\frac{3}{C-3};\quad \alpha_2=\frac{6Cx-5}{C-3};\\
    \ell&=\frac{D_1}{\alpha_1+1}=(C-3)x;\\
    D&=D_2-\ell(\ell+\alpha_2)=x(5-9x)>0;\\
    S&=\frac{D}{2\alpha_1+1}=x(5-9x)\frac{C-3}{C+3}.
    \end{align*}

Then,  
$$\frac{S_n-(C-3)xn}{\sqrt n}\Longrightarrow \mathcal{N}\left(0,x(5-9x)\frac{C-3}{C+3}\right).$$

For example, with $w_0=a=2,b_0=b=3$, we have
$$\frac{S_n-\frac45n}{\sqrt n}\Longrightarrow \mathcal{N}\Bigl(0,\frac7{50}\Bigr).$$

When $a+b\ge 4$, with $b=a+2$ or ($b=a+1$ and $w_0+b_0$ is even), the white never dominates, but equality of the two colors is possible (which is why~\eqref{bonneequation} fails).
Simulations seem to show that eventually, blacks definitely prevail with $\frac{S_n}{n(a+b-3)}\to \frac{a}{a+b}<\frac12$ and that the stated Central Limit Theorem is always valid.

When the Central Limit Theorem is established, the convergence in probability of 
 $\frac{S_n}{n(a+b-3)}$ to $\frac{a}{a+b}$ follows.
The question of almost sure convergence is natural, and, presumably, more delicate.

\def\refname{References}
\bibliographystyle{plain}
\bibliography{euler}

\end{document}